\documentclass[a4paper,reqno,12pt]{amsart}
\usepackage{verbatim}
\usepackage{amssymb,amsmath,amsthm}
\usepackage{amsfonts}
\usepackage{array}

\newcommand{\qbin}[2]{\genfrac{[}{]}{0pt}{}{#1}{#2}_q}

\newtheorem{theorem}{Theorem}

\theoremstyle{remark}

\begin{document}

\setcounter{page}{1}

\title[]{Two supercongruences related to multiple harmonic sums}

\author{Roberto~Tauraso}

\address{Dipartimento di Matematica, % \\
Universit\`a di Roma ``Tor Vergata'', % \\
via della Ricerca Scientifica,%\\
00133 Roma, Italy}
\email{tauraso@mat.uniromA1.it}

%11A07 Congruences; primitive roots; residue systems
%33C20 Generalized hypergeometric series,
%33B15 Gamma, beta and polygamma functions
%11S80 Other analytic theory (analogues of beta and gamma functions,p-adic integration
%11B65 Binomial coefficients; factorials;q-identities
%11B68   	Bernoulli and Euler numbers and polynomials

\date{\today}

\begin{abstract} 
Let $p$ be a prime and let $x$ be a $p$-adic integer. We provide two supercongruences for truncated series of the form 
$$\sum_{k=1}^{p-1} \frac{(x)_k}{(1)_k}\cdot
\frac{1}{k}\sum_{1\le j_1\le\cdots\le j_r\le k}\frac{1}{j_1^{}\cdots j_r^{}}\quad\mbox{and}\quad
\sum_{k=1}^{p-1} \frac{(x)_k(1-x)_k}{(1)_k^2}\cdot
\frac{1}{k}\sum_{1\le j_1\le\cdots\le j_r\le k}\frac{1}{j_1^{2}\cdots j_r^{2}}.$$
\end{abstract}

\maketitle

\section{Introduction and main results}

In \cite[Theorem 1.1]{Ta:10} and \cite[Theorem 7]{Ta:12} we showed that for any prime $p\not=2$,
$$\sum_{k=1}^{p-1} \frac{\binom{2k}{k}}{k 4^k}
\equiv_{p^3} -H_{(p-1)/2}\quad\text{and}\quad
\sum_{k=1}^{p-1} \frac{\binom{2k}{k}^2}{k 16^k}
\equiv_{p^3} -2H_{(p-1)/2}$$
where $H_n^{(t)}=\sum_{j=1}^n\frac{1}{j^t}$ is the $n$-th harmonic number of order $t\geq 1$.
Here we present two extensions of such congruences which involves the (non-strict) multiple harmonic sums
$$S_n(t_1,\dots,t_r):=\sum_{1\le j_1\le\cdots\le j_r\le n}\frac{1}{j_1^{t_1}\cdots j_r^{t_r}}$$
with $t_1,t_2,\dots, t_r$ positive integers. For the sake of brevity, if $t_1 = t_2 = \dots = t_r = t$ we write
$S_n(\{t\}^r)$. 

Let $(x)_n:=x(x+1)\cdots (x+n-1)$ be the Pochhammer symbol, and let $B_n(x)$ be the $n$-th Bernoulli polynomial. For any prime $p$,
$\mathbb{Z}_p$ denotes the ring of all $p$-adic integers and 
$\langle \cdot\rangle_p$  is the least non-negative residue modulo $p$ of the $p$-integral argument.

\begin{theorem}\label{MT} Let $p$ be a prime, $x\in\mathbb{Z}_p$ and $r\in\mathbb{N}$.
Let $s:=(x+\langle-x\rangle_p)/p$.

\noindent i) If $p>r+3$ then
\begin{align}\label{SI1}
\sum_{k=1}^{p-1} \frac{(x)_k}{(1)_k}\cdot
\frac{S_k(\{1\}^r)}{k}
&\equiv_{p^2}
-H^{(r+1)}_{\langle-x\rangle_p}-(-1)^rsp B_{p-r-2}(x).
\end{align}

\noindent ii) If $p>2r+3$ then 
\begin{align}\label{SI2}
\sum_{k=1}^{p-1} \frac{(x)_k(1-x)_k}{(1)_k^2}\cdot
\frac{S_k(\{2\}^r)}{k}
&\equiv_{p^3}
-2H^{(2r+1)}_{\langle-x\rangle_p}-2(2r+1)sp H^{(2r+2)}_{\langle-x\rangle_p}
\nonumber \\&\qquad\qquad
+\frac{2s(1+3sr+2sr^2)}{2r+3}\,p^2 B_{p-2r-3}(x).
\end{align}
\end{theorem}

Note that, when $r=0$, both \eqref{SI1} and \eqref{SI2} have been established by Zhi-Hong Sun in \cite{Sunzh:15}. Moreover, for the special value $x=1/2$, 
\eqref{SI1} and \eqref{SI2} yield
\begin{equation}\label{CI1}
\sum_{k=1}^{p-1} \frac{\binom{2k}{k}}{k 4^k}\cdot S_k(\{1\}^r)
\equiv_{p^2}
\begin{cases}
-H_{(p-1)/2}^{(r+1)}
&\text{if $r\equiv_2 0$},\vspace{3mm}\\
\frac{2^{r+2}-1}{2(r+2)}\,p B_{p-r-2} 
&\text{if $r\equiv_2 1$,}
\end{cases}
\end{equation}
and
\begin{equation}\label{CI2}
\sum_{k=1}^{p-1} \frac{\binom{2k}{k}^2}{k 16^k}\cdot
S_k(\{2\}^r)
\equiv_{p^3}
-2H_{(p-1)/2}^{(2r+1)}
-\frac{r(2^{2r+3}-1)}{2}\,p^2B_{p-2r-3}.
\end{equation}
For $r=1$, the congruence \eqref{CI2}  proves the conjecture \cite[Conjecture 5.3]{Sunzw:15}. 

In the last section we provide $q$-analogs of two binomial identities related to the congruences \eqref{SI1} and  \eqref{SI2}. 

\section{Proof of \eqref{SI1} in Theorem \ref{MT}}
By taking the partial fraction expansion of the rational function
$$x\to \frac{(x)_k}{(x)_n}$$
with $0\leq k<n$, we find
\begin{equation}\label{PDF1}
\sum_{k=0}^{n-1} \frac{(x)_k}{(1)_k}\cdot a_k=
(x)_n\sum_{j=0}^{n-1}\frac{(-1)^j T_j}{j!(n-1-j)!}\cdot \frac{1}{x+j}
\end{equation}
where $T_j$ is the binomial transform of the sequence $a_k$,
$$T_j:=\sum_{k=0}^j(-1)^k\binom{j}{k}\cdot a_k.$$
It is easy to see from \eqref{PDF1} that if $a_0,\dots,a_{p-1},x\in \mathbb{Z}_p$ then
\begin{equation}\label{A1}
\sum_{k=0}^{p-1} \frac{(x)_k}{(1)_k}\cdot a_k\equiv_{p} T_{{\langle-x\rangle_p}}.
\end{equation}%+sp\sum_{j=1, j\not={\langle-x\rangle_p}}^{p-1}\frac{T_j}{j-\langle-x\rangle_p}
In order to show  \eqref{SI1} we introduce the function
$$G_n^{(r)}(x):=\sum_{k=1}^{n} \frac{(x)_k}{(1)_k}\cdot S_k(\{1\}^r).$$
We have that
$$G_n^{(0)}(x)=\frac{(1+x)_{n}}{(1)_{n}}-1$$
and $S_{k}(\{1\}^r)= S_{k-1}(\{1\}^r)+S_{k}(\{1\}^{r-1})/k$ implies
\begin{equation}\label{G1}
G_n^{(r)}(x)=\frac{(1+x)_{n}}{(1)_{n}}\cdot S_{n}(\{1\}^r)-\frac{G_n^{(r-1)}(x)}{x}.
\end{equation}
Moreover
$$F_n^{(r)}(x+1)-F_n^{(r)}(x)=\frac{G_n^{(r)}(j)}{x}$$
where
$$F_n^{(r)}(x):=\sum_{k=1}^{n} \frac{(x)_k}{(1)_k}\cdot
\frac{S_k(\{1\}^r)}{k}.$$
Then, for any positive integer $m$,
\begin{equation}\label{F1}
F_n^{(r)}(x+m)-F_n^{(r)}(x)=
\sum_{j=0}^{m-1}\frac{G_n^{(r)}(x+j)}{x+j}.
\end{equation}
By \eqref{G1}, for $u=1,\dots,n$
$$G_n^{(r)}(-u)=\frac{G_n^{(r-1)}(-u)}{u}=\cdots=\frac{G_n^{(0)}(-u)}{u^r}=-\frac{1}{u^r}.$$
Hence by letting $x=-n$ and $m=n$ in \eqref{F1} we obtain the known identity (see \cite{He:99})
\begin{equation}\label{He}
\sum_{k=1}^n (-1)^k\binom{n}{k}\frac{S_k(\{1\}^r)}{k}=-H_n^{(r+1)}.
\end{equation}
Thus, for $a_k=\frac{S_k(\{1\}^r)}{k}$, we have that $T_j=-H_j^{(r+1)}$, and by \eqref{A1}, we already have the modulo $p$ version of \eqref{SI1}.
 
\begin{proof}[Proof of \eqref{SI1} in Theorem \ref{MT}] 
Since $sp=x+\langle-x\rangle_p$ it follows that
$$G_{p-1}^{(0)}(x) =\frac{(1+x)_{p-1}}{(1)_{p-1}}-1
\equiv_{p^2} \frac{sp}{x}-1.$$
By \cite[Theorem 1.6]{Zh:07}, 
$S_{p-1}(\{1\}^r)\equiv_p 0$ and therefore
$$G_{p-1}^{(r)}(x)\equiv_{p^2}
-\frac{G_{p-1}^{(r-1)}(x)}{x}\equiv_{p^2}\cdots
\equiv_{p^2} (-1)^r\frac{G_{p-1}^{(0)}(x)}{x^{r}}
\equiv_{p^2} \frac{(-1)^r sp}{x^{r+1}}-\frac{(-1)^r}{x^{r}}.
$$
Moreover
\begin{align*}
F_{p-1}^{(r)}(sp)&=
\sum_{k=1}^{p-1} \frac{(sp)_k}{(1)_k}\cdot
\frac{S_k(\{1\}^r)}{k}\equiv_{p^2}
\sum_{k=1}^{p-1} \frac{sp}{k}\cdot
\frac{S_k(\{1\}^r)}{k}\\
&=spS_{p-1}(\{1\}^{r},2))
\equiv_{p^2} 
spB_{p-r-2}
\end{align*}
where we used 
$S_{p-1}(\{1\}^r,2)\equiv_p B_{p-r-2}$ (see \cite[Theorem 4.5]{HHT:14}).

Finally, by \eqref{F1},
\begin{align*}
F_{p-1}^{(r)}(x)&\equiv_{p^2}\sum_{j=0}^{\langle-x\rangle_p-1}
\left(\frac{(-1)^r}{(x+j)^{r+1}}-\frac{(-1)^rsp}{(x+j)^{r+2}}\right)+spB_{p-r-2}\\
&\equiv_{p^2}-\sum_{j=1}^{\langle-x\rangle_p}
\frac{1}{(j-sp)^{r+1}}-sp\sum_{j=1}^{\langle-x\rangle_p}\frac{1}{j^{r+2}}+spB_{p-r-2}\\
&\equiv_{p^2}-H_{\langle-x\rangle_p}^{(r+1)}-
(r+2)spH_{\langle-x\rangle_p}^{(r+2)}+spB_{p-r-2}\\
&\equiv_{p^2}-H_{\langle-x\rangle_p}^{(r+1)}-(-1)^r spB_{p-r-2}\\
\end{align*}
where the last step uses the following congruence: for $2\leq t<p-1$ 
\begin{equation}\label{Hp}
H_{\langle-x\rangle_p}^{(t)}
\equiv_p \sum_{j=1}^{\langle-x\rangle_p} j^{p-1-t}=
\frac{B_{p-t}(\langle-x\rangle_p+1)-B_{p-t}}{p-t}
\equiv_p
(-1)^t\frac{B_{p-t}(x)-B_{p-t}}{t}
\end{equation}
which is an immediate consequence of \cite[Lemma 3.2]{Sunzh:00}.
\end{proof} 

\section{Proof of \eqref{SI2} in Theorem \ref{MT}}

We follow a similar strategy as outlined in the previous section.
We start by considering the partial fraction decomposition of the rational function
$$x\to\frac{(x)_k(1-x)_k}{(x)_n(1-x)_n}$$
with $0\leq k<n$. We have that
\begin{equation}\label{PDF2}
\sum_{k=0}^{n-1} \frac{(x)_k(1-x)_k}{(1)_k^2}\cdot a_k=
(x)_n(1-x)_n
\sum_{j=0}^{n-1}\frac{(-1)^j A_j}{(n+j)!(n-1-j)!}\left(\frac{1}{x+j}+\frac{1}{1-x+j}\right)
\end{equation}
where
$$A_j:=\sum_{k=0}^j(-1)^k\binom{j}{k}\binom{j+k}{k}\cdot a_k.$$
For $n\to \infty$, if the series is convergent, the identity \eqref{PDF2} becomes
$$\sum_{k=0}^{\infty} \frac{(x)_k(1-x)_k}{(1)_k^2}\cdot a_k=
\frac{\sin(\pi x)}{\pi}
\sum_{j=0}^{\infty}(-1)^j A_j\left(\frac{1}{x+j}+\frac{1}{1-x+j}\right).
$$
In many cases the transformed sequence $A_j$ has a \textit{nice} formula.
For example if $a_k=1/(k+z)$ then
$$A_j=\frac{(1-z)_j}{(z)_{j+1}}$$
and for $x=z=1/2$ we recover this series representations the Catalan's constant $G=\sum_{j=0}\frac{(-1)^j}{(2j+1)^2}$:
$$\sum_{k=0}^{\infty} \frac{\binom{2k}{k}^2}{(2k+1)16^k}=\frac{1}{2}\sum_{k=0}^{\infty} \frac{(1/2)_k^2}{(1)_k^2(k+1/2)}=
\frac{1}{2\pi}
\sum_{j=0}^{\infty}(-1)^j \frac{4}{(1/2+j)^2}=
\frac{8G}{\pi}.$$

As regards congruences we have the following result.

\begin{theorem}
Let $p$ be a prime with $a_0,\dots,a_{p-1},x\in \mathbb{Z}_p$. Then
\begin{equation}\label{A2}
\sum_{k=0}^{p-1} \frac{(x)_k(1-x)_k}{(1)_k^2}\cdot a_k
\equiv_{p^2} 
A_{\langle-x\rangle_p}+s(A_{p-1-\langle-x\rangle_p}-A_{\langle-x\rangle_p})
\end{equation}
For $x=1/2$ and $p>2$ then 
$$
\sum_{k=0}^{p-1} \frac{\binom{2k}{k}^2}{16^k}\cdot a_k
\equiv_{p^2} 
A_{(p-1)/2}.
$$
\end{theorem}
\begin{proof}  Rearranging \eqref{PDF2} in a convenient way, we have
$$
\sum_{k=0}^{p-1} \frac{(x)_k(1-x)_k}{(1)_k^2}\cdot a_k=
\frac{(x)_p(1-x)_p}{(1)_p^2}\binom{2p-1}{p-1}^{-1}
\sum_{j=0}^{p-1}(-1)^j\binom{2p-1}{p+j} A_j\left(\frac{p}{x+j}+\frac{p}{1-x+j}\right).
$$
If $0\leq k\leq j\leq p-1$ then $A_{p-1-j}\equiv_p A_{j}$ because
\begin{align*}
\binom{p-1-j}{k}\binom{p-1-j+k}{k}
&=\frac{(p-1-j)\cdots (p-j-k)(p-1-j+k)\cdots (p-j) }{(k!)^2}\\&\equiv_p 
\frac{(j+1)\cdots (j+k)(j-k+1)\cdots j }{(k!)^2}
=\binom{j}{k}\binom{j+k}{k}.
\end{align*}
Thus, since $\langle -x\rangle_p+\langle-(1-x)\rangle_p=p-1$, it follows that
\begin{align*}
\sum_{j=0}^{p-1}(-1)^j\binom{2p-1}{p+j} \frac{pA_j}{x+j}
&\equiv_{p^2}\sum_{j=0}^{\langle-x\rangle_p-1} \frac{pA_j}{x+j}+(-1)^{\langle-x\rangle_p}\binom{2p-1}{p+{\langle-x\rangle_p}} \frac{A_{\langle-x\rangle_p}}{s}+\sum_{j=\langle-x\rangle_p+1}^{p-1} \frac{pA_j}{x+j}\\
&\equiv_{p^2}\sum_{j=0}^{\langle-x\rangle_p-1} \frac{pA_j}{x+j}+(-1)^{\langle-x\rangle_p}\binom{2p-1}{p+{\langle-x\rangle_p}} \frac{A_{\langle-x\rangle_p}}{s}
-\!\!\!\!\!\!\sum_{j=0}^{\langle-(1-x)\rangle_p-1}\!\!\!
\frac{pA_{p-1-j}}{1-x+j}.
\end{align*}
Therefore 
\begin{align*}
\sum_{j=0}^{p-1}(-1)^j\binom{2p-1}{p+j} A_j\left(\frac{p}{x+j}+\frac{p}{1-x+j}\right)
&\equiv_{p^2}(-1)^{\langle-x\rangle_p}\binom{2p-1}{p+{\langle-x\rangle_p}} \left(\frac{A_{\langle-x\rangle_p}}{s} +\frac{A_{\langle-(1-x)\rangle_p}}{1-s}\right).
\end{align*}
Finally, by using
\begin{align*}
&\binom{2p-1}{p-1}  \equiv_{p^3} 1,\\
&\binom{2p-1}{p+j}  \equiv_{p^2} (-1)^{j}\left(1-2p H_{j}\right),\\
&\frac{(x)_p(1-x)_p}{(1)_p^2}\equiv_{p^2} s(1-s)\left(1+2p H_{\langle-x\rangle_p}\right),
\end{align*}
we are done. For $x=1/2$ it suffices to note that  
$$\langle-x\rangle_p=(p-1)/2= p-1-\langle-x\rangle_p.$$
\end{proof}

As an application of the previous theorem, we note that when $a_k=1$ then $A_j=(-1)^j$, and, by \eqref{A2}, it follows  that 
$$
\sum_{k=0}^{p-1} \frac{(x)_k(1-x)_k}{(1)_k^2}
\equiv_{p^2} (-1)^{\langle -x\rangle_p}$$
which has been established in \cite[Corollary 2.1]{Sunzh:14}.
Another example worth to be mentioned is 
$a_k=1/k^r$ for $k\geq 1$ (and $a_0=0$). Then by \cite[Theorem 1]{Pr:10}
$$A_j=-\sum_{1\cdot k_1+3\cdot k_3+\dots= r}
\frac{2^{k_1+k_3+\dots} (H_j^{(1)})^{k_1}(H_j^{(3)})^{k_3}\cdots}{1^{k_1}3^{k_3}\cdots k_1!k_3!\cdots}.$$
 
Now we consider the case $a_k=S_k(\{2\}^r)/k$.
Let
$$G_n^{(r)}(x):=\sum_{k=1}^{n} \frac{(x)_k(-x)_k}{(1)_k^2}\cdot S_k(\{2\}^r).$$
We have that
$$G_n^{(0)}(x)=\frac{(1+x)_{n}(1-x)_{n}}{(1)_{n}^2}-1,$$
and $S_{k}(\{2\}^r)= S_{k-1}(\{2\}^r)+S_{k}(\{2\}^{r-1})/k^2$ implies
\begin{equation}\label{G2}
G_n^{(r)}(x)=\frac{(1+x)_{n}(1-x)_{n}}{(1)_{n}^2}\cdot S_{n}(\{2\}^r)+\frac{G_n^{(r-1)}(x)}{x^2}.
\end{equation}
%which implies
%$$G_n^{(r)}(x)=
%\frac{(1+x)_{n}(1-x)_{n}}{(1)_{n}^2}
%\sum_{i=1}^r \frac{S_{n}(\{2\}^i)}{x^{2(r-i)}}+\frac{G_n^{(0)}(x)}{x^{2r}}$$
Moreover
$$F_n^{(r)}(x+1)-F_n^{(r)}(x)=\frac{2G_n^{(r)}(x)}{x}$$
where
$$F_n^{(r)}(x):=\sum_{k=1}^{n} \frac{(x)_k(1-x)_k}{(1)_k^2}\cdot
\frac{S_k(\{2\}^r)}{k}.$$
Hence
\begin{equation}\label{F2}
F_n^{(r)}(x+m)-F_n^{(r)}(x)=
2\sum_{j=0}^{m-1}\frac{G_n^{(r)}(x+j)}{x+j}.
\end{equation}

The next identity is a variation of \eqref{He} and it appears to be new.
\begin{theorem}
For any integers $n\geq 1$ and $r\geq 0$,
\begin{equation}\label{Ta}
\sum_{k=1}^{n} (-1)^k\binom{n}{k}\binom{n+k}{k}\frac{S_k(\{2\}^r)}{k}=-2H_{n}^{(2r+1)}.
\end{equation}
\end{theorem} 
\begin{proof} By \eqref{G2}, for $u=1,\dots,n$, 
$$G_n^{(r)}(-u)=\frac{G_n^{(r-1)}(-u)}{u^2}=\dots=\frac{G_n^{(0)}(-u)}{u^{2r}}=-\frac{1}{u^{2r}}.$$
Hence by letting $x=-n$ and $m=n$ in \eqref{F2}
\begin{align*}
\sum_{k=1}^{n} (-1)^k\binom{n}{k}\binom{n+k}{k}\frac{S_k(\{2\}^r)}{k}&=F^{(r)}_n(-n)=F_n^{(r)}(0)-2\sum_{j=0}^{n-1}\frac{G_n^{(r)}(-n+j)}{-n+j}\\
&=2\sum_{j=0}^{n-1}\frac{1}{(-n+j)^{2r+1}}=-2H_{n}^{(2r+1)}.
\end{align*}
\end{proof}

Thus by applying $\eqref{A2}$ we find a modulo $p^2$ version of $\eqref{SI2}$. 
A more refined reasoning will lead us to the $p^3$ congruence.

\begin{proof}[Proof of \eqref{SI2} in Theorem \ref{MT}] 
Since $sp=x+\langle-x\rangle_p$,
$$G_{p-1}^{(0)}(x) =\frac{(1+x)_{p-1}(1-x)_{p-1}}{(1)_{p-1}^2}-1
\equiv_{p^3} -\frac{s(1-s)p^2}{x^2}-1$$
By \cite[Theorem 1.6]{Zh:07}, 
$S_{p-1}(\{2\}^r)\equiv_p 0$ and therefore
$$G_{p-1}^{(r)}(x)\equiv_{p^3}
\frac{G_{p-1}^{(r-1)}(x)}{x^{2}}\equiv_{p^3}\cdots
\equiv_{p^3} \frac{G_{p-1}^{(0)}(x)}{x^{2r}}
\equiv_{p^3} -\frac{s(1-s)p^2}{x^{2r+2}}-\frac{1}{x^{2r}}.
$$
It follows that
\begin{align*}
F_{p-1}^{(r)}(sp)-F_{p-1}^{(r)}(x)
&\equiv_{p^3} 
2\sum_{j=0}^{\langle-x\rangle_p-1}\frac{G_{p-1}^{(0)}(x+j)}{(x+j)^{2r+1}}\\
&\equiv_{p^3}
-2s(1-s)p^2\sum_{j=1}^{\langle-x\rangle_p}\frac{1}{j^{2r+3}}-2\sum_{j=0}^{\langle-x\rangle_p-1}\frac{1}{(x+j)^{2r+1}}.
\end{align*}
By \eqref{Hp}
$$\sum_{j=1}^{\langle-x\rangle_p}\frac{1}{j^{2r+3}}=H_{\langle-x\rangle_p}^{(2r+3)}
\equiv_p
-\frac{B_{p-2r-3}(x)-B_{p-2r-3}}{2r+3}.$$
Moreover
\begin{align*}
F_{p-1}^{(r)}(sp)&=
\sum_{k=1}^{p-1} \frac{(sp)_k(1-sp)_k}{(1)_k^2}\cdot
\frac{S_k(\{2\}^r)}{k}\\
&\equiv_{p^3}
\sum_{k=1}^{p-1} \frac{sp(k-sp)}{k^2}\cdot
\frac{S_k(\{2\}^r)}{k}\\
&=sp\sum_{k=1}^{p-1}
\frac{S_k(\{2\}^r)}{k^2}
-p^2s^2\sum_{k=1}^{p-1}
\frac{S_k(\{2\}^r)}{k^3}\\
&=spS_{p-1}(\{2\}^{r+1}))
-p^2s^2S_{p-1}(\{2\}^{r},3))\\
&\equiv_{p^3} 
sp\frac{2pB_{p-2r-3}}{2r+3}
+p^2s^22rB_{p-2r-3}\\
&\equiv _{p^3}
\frac{2sp^2(1+sr(2r+3))B_{p-2r-3}}{2r+3}
\end{align*}
where we used
$$\frac{(sp)_k(1-sp)_k}{(1)_k^2}
=\frac{sp(k-sp)}{k^2}\cdot \frac{(1+sp)_{k-1}(1-sp)_{k-1}}{ (1)_{k-1}^2}\equiv_{p^3}\frac{sp(k-sp)}{k^2}$$
and the congruences
$$S_{p-1}(\{2\}^r)\equiv_{p^2} \frac{2pB_{p-2r-1}}{2r+1}\quad\text{and}\quad
S_{p-1}(\{2\}^r,3)\equiv_p -2rB_{p-2r-3}.$$
which have been established in \cite[Theorem 1.6]{Zh:07} in \cite[Theorem 4.1]{HHT:14} respectively.
Finally,
\begin{align*}
F_{p}^{(r)}(x)
&\equiv_{p^3} 
\frac{2sp^2(1+sr(2r+3))B_{p-2r-3}}{2r+3}
-\frac{2s(s-1)p^2(B_{p-2r-3}(x)-B_{p-2r-3})}{2r+3}\\
&\qquad\qquad+2\sum_{j=0}^{\langle-x\rangle_p-1}\frac{1}{(x+j)^{2r+1}}\\
&\equiv_{p^3}
2\sum_{j=0}^{\langle-x\rangle_p-1}\frac{1}{(x+j)^{2r+1}}
+\frac{2s(1-s)}{2r+3}p^2B_{p-2r-3}(x)
\\&\qquad\qquad
+\frac{2s^2(r+1)(2r+1)}{2r+3}p^2 B_{p-2r-3}
\end{align*}
\end{proof}

We observe  that \eqref{CI2} follows by letting $x=1/2$. Then $\langle-x\rangle_p-1=(p-1)/2$, $B_{2n}(1/2)=(2^{1-2n}-1)B_{2n}$ and for $p-4>t>1$
$$H^{(t)}_{(p-1)/2}\equiv
\begin{cases}
\frac{t(2^{t+1}-1)}{2(t+1)}\,p B_{p-t-1} \pmod{p^2}
&\text{if $t\equiv_2 0$},\vspace{3mm}\\
-\frac{(2^{t}-2)}{t}\, B_{p-t}    \qquad \pmod{p}
&\text{if $t\equiv_2 1$}.
\end{cases}$$
see \cite[Theorem 5.2]{Sunzh:00}.

\section{Final remarks: $q$-analogs of \eqref{He} and \eqref{Ta}}

It is interesting to note that identities \eqref{He} and \eqref{Ta} have both a $q$-version (the first one appears in \cite{Pr:00}).

\begin{theorem} For any integers $n\geq 1$ and $r\geq 0$,
\begin{equation}\label{Heq}
\sum_{k=1}^n(-1)^k\qbin{n}{k}q^{\binom{k}{2}-(n-1)k}\cdot \frac{S_k(\{1\}^r;q)}{1-q^k}
=-\sum_{k=1}^n\frac{q^{rk}}{(1-q^k)^{r+1}}
\end{equation}
and 
\begin{equation}\label{Taq}
\sum_{k=1}^n(-1)^k\qbin{n}{k}\qbin{n+k}{k} q^{\binom{k}{2}-(n-1)k} 
\cdot \frac{S_k(\{2\}^r;q)}{1-q^k}
=-\sum_{k=1}^n\frac{(1+q^k)q^{rk}}{(1-q^{k})^{2r+1}}
\end{equation}
where $\qbin{m}{k}$ is the  Gaussian binomial coefficient 
$$\qbin{m}{k}=\left\{
\begin{array}{ll}
\frac{(1-q^m)(1-q^{m-1})\cdots (1-q^{m-k+1})}{(1-q^k)(1-q^{k-1})\cdots (1-q)} 
&\mbox{if $0\leq k\leq m$},\\[3pt]
0 &\mbox{otherwise},
\end{array}\right.$$
and
$$S_n(t_1,\dots,t_r;q):=\sum_{1\le j_1\le\cdots\le j_r\le n}\frac{q^{j_1+\dots +j_r}}{(1-q^{j_1})^{t_1}\cdots (1-q^{j_r})^{t_r}}.$$
\end{theorem}

\begin{proof} We show \eqref{Ta} and we leave the proof of other one to the interested reader. The procedure is quite similar to the one given for the corresponding ordinary identity \eqref{Ta}.  Let
$$G_n^{(r)}(u):=\sum_{k=1}^n(-1)^k\qbin{u}{k}\qbin{u+k-1}{k} q^{\binom{k}{2}-(u-1)k} 
\cdot S_k(\{2\}^r;q).$$
Then for $u=1,\dots,n$, $G_n^{(0)}(u)=-1$ and
$$G_n^{(r)}(u)=\frac{q^u G_n^{(r-1)}}{(1-q^u)^2}=\dots=\frac{q^{ru }G_n^{(0)}(u)}{(1-q^u
)^{2r}}=-\frac{q^{ru }}{(1-q^u)^{2r}}.$$
Moreover
$$F_n^{(r)}(u)-F_n^{(r)}(u-1)=\frac{(1+q^u)G_n^{(r)}(u)}{(1-q^u)}
=-\frac{(1+q^u)q^{ru }}{(1-q^u)^{2r+1}}$$
where
$$F_n^{(r)}(u):=\sum_{k=1}^n(-1)^k\qbin{u}{k}\qbin{u+k}{k} q^{\binom{k}{2}-(u-1)k} 
\cdot S_k(\{2\}^r;q).$$
Thus, since $F_n^{(0)}(n)=0$,
\begin{align*}
F_n^{(r)}(n)&=\sum_{u=1}^n\frac{(1+q^u)G_n^{(r)}(u)}{(1-q^u)}+F_n^{(0)}(n)
=-\sum_{u=1}^n\frac{(1+q^u)q^{ru}}{(1-q^{u})^{2r+1}}
\end{align*}
and the proof is complete.
\end{proof}

\end{document}